\documentclass[aos,preprint,noinfoline]{imsart}
\RequirePackage{amsthm,amsmath,amsfonts,amssymb}
\RequirePackage[authoryear]{natbib} %% uncomment this for author-year bibliography
\RequirePackage[colorlinks,citecolor=blue,urlcolor=blue]{hyperref}
\RequirePackage{graphicx}
\usepackage{enumerate, bbm}
\usepackage{dsfont}
\usepackage{mathrsfs}
\usepackage{color}
\usepackage{graphicx}

\startlocaldefs
\startlocaldefs

\newcommand{\PP}{\mathbb{P}}
\newcommand{\EE}{{\mathbb{E}}}

\newcommand{\lle}{\, \, {\lesssim}\, \, }

\newcommand{\FF}{\mathcal{F}}
\newcommand{\A}{\mathcal{A}}

\newcommand{\RR}{\mathbb{R}}

\newcommand{\NN}{\mathbb{N}}

\newcommand{\BB}{\mathcal{B}}

\newcommand{\TT}{\mathcal{T}}

\newcommand{\XX}{\mathcal{X}}

\newcommand{\YY}{\mathcal{Y}}

\newcommand{\ZZ}{\mathbb{Z}}

\newcommand{\br}[1]{\left[#1 \right]}
\newcommand{\ind}[1]{\mathbbm{1}_{#1}}

\newcommand{\MMM}{\mathcal{M}}
\newcommand{\NNN}{\mathcal{N}}

\newcommand{\given}{\, |\, }

\newcommand{\llogn}{\lfloor \log n\rfloor}

\newcommand{\cG}{\ensuremath{\mathcal G}}
\newcommand{\cF}{\ensuremath{\mathcal F}}

\newcommand{\reals}{\ensuremath{{\mathbb R}}}
\newcommand{\naturals}{\ensuremath{{\mathbb N}}}

\newcommand{\MMLE}{\hat\mu^{\text{MLE}}}
\theoremstyle{plain}
\newtheorem{thm}{Theorem}[section]
\newtheorem{lem}{Lemma}[section]
\newtheorem{prop}[thm]{Proposition}
\newtheorem{cor}[thm]{Corollary}

\newtheorem{assumption}{Assumption}
\theoremstyle{remark}
\newtheorem{rem}[thm]{Remark}

\newtheorem{df}[thm]{Definition}

\endlocaldefs

\begin{document}

\begin{frontmatter}
\title{Minimax rates without the fixed sample size assumption}
\runtitle{Minimax rates without the fixed $n$ assumption}

\begin{aug}
\author[a]{\fnms{Alisa} \snm{Kirichenko}\ead[label=e1]{alice.kirichenko@warwick.ac.uk}}
\and
\author[b]{\fnms{Peter} \snm{Gr\"{u}nwald}\ead[label=e2]{peter.grunwald@cwi.nl}}

\address[a]{Department of Statistics, 
University of Warwick, 
\printead{e1}}

%\address[b]{CWI
%}

\address[b]{CWI and Mathematical Institute, Leiden University,
\printead{e2}}

\runauthor{Kirichenko and Grunwald}

\end{aug}

\begin{abstract}
We generalize the notion of minimax convergence rate. In contrast to
the standard definition, we do not assume that the sample size is
fixed in advance. Allowing for varying sample size results in {\em
 time-robust\/} minimax rates and estimators. These can be either strongly
adversarial, based on the worst-case over all sample sizes, or weakly
adversarial, based on the worst-case over all stopping times. We
show that standard and time-robust rates usually differ by at most a
logarithmic factor, and that for some (and we conjecture for all)
exponential families, they differ by exactly an iterated logarithmic
factor. In many situations, time-robust rates are arguably more
natural to consider. For example, they allow us to simultaneously
obtain strong model selection consistency and optimal estimation
rates, thus avoiding the ``AIC-BIC dilemma''.
\end{abstract}
%
%\begin{keyword}
%\kwd{minimax rates}
%\kwd{asymptotic theory}
%\kwd{optional stopping}
%\kwd{AIC-BIC dilemma}
%\end{keyword}

\end{frontmatter}
\thispagestyle{empty}
\numberwithin{equation}{section}
\numberwithin{equation}{section}
\section{Introduction}

Minimax rates are an essential tool for evaluation and comparison of
estimators in a wide variety of applications. Classic references on
the topic include, among many others, \cite{tsybakov}, 
\cite{wasserman} and \cite{vdv}. For a fixed sample size $n$, the
standard minimax rate is computed by first taking the supremum of the
expected loss over all parameters (distributions) in the model for
each estimator, and then minimizing this value over all possible
estimators. Here, we consider a natural extension of this setting in
which data comes in sequentially and one does not know $n$ in advance:
instead of considering $n\geq 1$ fixed, we include it in the
worst-case analysis.

At first it may seem that such time-robustness trivializes the
problem: a naive approach would be to take $n$ as a parameter just
like the distribution and compute the supremum of the expected loss
over all sample sizes and all distributions in the model.
In most cases the supremum would then be trivially attained for sample
size one, since the precision of an
estimator tends to get better with the increase of the sample
size. Therefore, another approach has to be taken. We manage to give
meaningful definitions by rewriting the standard definition in terms of a ratio. The precise new definitions, given in Section
\ref{sec:new_def}, come in two forms: {\em weakly adversarial}, in
which we take the $\sup$ (worst-case) over all stopping times; and
{\em strongly adversarial}, in which we take the $\sup$ over all
sample sizes. In general, the weakly adversarial minimax rate cannot
be larger than the strongly adversarial one. The weakly adversarial
setting corresponds to what has recently been called the {\em always
 valid\/} (sometimes also ``anytime-valid'') setting for confidence
intervals and testing \citep{howard}: at any point in time $n$, Nature
can decide whether or not to stop generating data and present the data
so far for analysis, using a rule that can take into account both past
data and the true distribution. This can be seen as a form of minimax
analysis under `optional stopping'. Note however that in the standard
interpretation (e.g. in the Bayesian literature) of optional stopping, 
stopping rules are assumed independent of the underlying distribution
$\PP_{\theta}$, whereas here Nature is more powerful, her stopping
rule being allowed to depend on $\theta$. Nature is even more powerful
in the second, strongly adversarial setting that we consider. Here, 
Nature can be thought of as generating a very large sample of data and
then simply producing the $n$ so that the initial sub-sample up to
size $n$ is as misleading as possible. While one may argue which of
these two settings is more appropriate, our initial results show that
in some cases they lead to the same rates, and we conjecture that the
rates coincide more generally.

%\paragraph{Motivation}
\smallskip
\noindent{\bf Motivation.}
One advantage of the new definitions is that, in some contexts, they
may be more natural. Of course, minimax approaches are truly optimal
in the zero-sum game setting in which Statistician plays against
Nature, Nature being a player that actively determines the parameters
of the problem in an adversarial manner. In practice, one is
interested in minimax estimators and rates not because one really
thinks that Nature will actually be adversarial in this way, but
simply because one wants to be robust against whatever might
happen. But if one wants to be robust against whatever might happen, 
then it seems natural to be robust not just for all parameters, but
also for sample sizes: in modern practice, the data analyst is often
presented with a fixed sample of a particular size, and she has no
control whatsoever on how exactly that sample size was
determined. Time-robust minimax optimal estimators are robust in this
situation --- one might of course argue `time will not be determined
by an adversary!' but this is no different from arguing `the true
$\theta$ will not be determined by an adversary!': once one takes a
worst-case approach at all, it makes sense to include time as
well. Moreover, even in the setting of controlled experiments such as
clinical trials, where the statistician is normally supposed to determine the sample
size in advance, early stopping and the like might happen for
reasons outside of the statistician's control, see e.g.\ \citep{medical} and references therein. As
such, the time-robust minimax setting nicely fits in recent work
promoting {\em always-valid\/} confidence intervals
\citep{howard, pace2019likelihood} and testing {\em safe under optional
 continuation\/} \citep{GrunwaldHK19} as a generic, more robust
replacement of traditional testing and confidence.

\begin{sloppypar}
Given the fact that time-robustness is a natural mode of analysis, it
is perhaps not so surprising, that the somewhat disturbing conflict
between consistency and rate optimality in standard estimation theory
known as {\em AIC-BIC dilemma\/} \citep{yang, tim, almost}, quite simply
disappears under the novel definition of minimax rate. We discuss this
motivating application at length towards the end of the paper, in
Section~\ref{sec:main}.
\end{sloppypar}

%\paragraph{Results}
\smallskip
\noindent{\bf Results.}
We provide several results comparing the time-robust to the standard
minimax rates. First, in Theorem~\ref{general} we show that for most
estimation problems the strongly adversarial minimax rate goes up by
at most a logarithmic factor. A natural question arises: is there an
estimation problem, for which time-robust minimax rates and standard
minimax rates do not coincide? The answer is positive: in
Theorem~\ref{thm:exp} we show that, under the standard squared error
loss, both the weakly and the strongly adversarial time-robust rates for
estimating a parameter in the Gaussian location family are equal to
$n^{-1}\log\log n$, while the standard minimax rate for this problem
is $n^{-1}$. The proof for the upper bound easily extends to most
standard multivariate exponential families, as we show in
Theorem~\ref{thm:upperbound}, and we conjecture the lower bound
extends as well.

These results originate from the law of iterated logarithm. To get an
intuition, consider the maximum likelihood estimator for the mean of a
one-dimensional Gaussian distribution with known variance. The
estimator is simply equal to the sample average. By the law of
iterated logarithm (see, for instance, \cite{lil}) the squared
distance between the sample average and the truth is of order
$n^{-1}\log\log n$ infinitely often with probability one. Therefore, 
for a suitable stopping rule the expected loss will also be of at
least the same order. While a lower bound on the rate for the MLE is
thus easy to determine, it turns out to be considerably more difficult
to show this lower bound for arbitrary estimators --- despite the
simple Gaussian location setting, this required new techniques. One
reason is that we must allow for arbitrary estimators, and these can
depend on the data in tricky ways. For example, one might change one's
estimate if the empirical average on the first half of the data is
more than a constant times $\sqrt{n^{-1}\log\log n}$ from the
empirical average on the second half. Since we show that no estimator
(decision rule) at all can beat $n^{-1}\log\log n$, we may think of
Theorem~\ref{thm:exp} as a {\em decision-theoretic law of the iterated
 logarithm}.

\begin{sloppypar}
The proofs for the upper bound on the strongly adversarial rate are
based on finite-time laws of the iterated logarithm based on
nonnegative supermartingales, a technique initially proposed by
\cite{darling}, and recently extended by
e.g. \cite{balsubramani, howard}. To adjust these techniques to our
strongly adversarial setting, we use two fundamental results from
\cite{shafer2011test} that link nonnegative supermartingales to
$p$-values and so-called $E$-values
\citep{vovk2019combining, GrunwaldHK19}.
\end{sloppypar}

The remainder of the paper is organized as follows. We give
the necessary measure-theoretic background in Section
\ref{sec:background}. Section \ref{sec:old_def} recalls the standard
definition of minimax rates. Section
\ref{sec:new_def} extends this definition to time-robust
minimax rates. Section \ref{sec:main} contains the main results of the paper, 
and the AIC-BIC example showing how the new definitions
can be used in the context of combined model selection and
estimation. We provide a short discussion in Section
\ref{sec:discussion}. All proofs are given in Section
\ref{sec:proofs}, with some details deferred to the appendix.
\section{Basic definitions}
\label{sec:def}
\subsection{Background on measure theory; notation}
\label{sec:background}

Let $\XX$ be a topological space endowed with Borel sigma-algebra $\BB$. Consider a probability space $(\Omega, \A, \PP)$. We say that a random variable $X:\Omega\to\XX$ is measurable on $(\Omega, \A)$ if $X^{-1}(B)=\{\omega\in\Omega:X(\omega)\in B\}\in\A$ for every Borel set $B\in\BB$. 
For a random variable $X:\Omega\to\XX$ let $\sigma(X)$ be a {\it sigma-algebra generated by $X$}, defined as the smallest sigma-algebra such that $X$ is measurable on $(\Omega, \sigma(X))$. Similarly, for a sequence of random variables $X_1, \dots, X_n:\Omega\to\XX$ denote the {\it sigma-algebra generated by $X_1, \dots, X_n$} by $\sigma(X_1, \dots, X_n)$. A {\it filtration} $\FF=(\FF_n)_{n\in\NN}$ is defined as a nondecreasing family of sigma-algebras.
We say that a random variable $\tau:\Omega\to\NN$ is a {\it stopping time} with respect to $\FF$, if $\{\omega:\tau(\omega)\leq n\}\in\FF_n$ for all $n\in\NN$. For more background on measure theory and stopping times see, for instance, \cite{olav} (Chapters 1, 2, and 7). 
 
 We write $a_n\lle b_n$, when there exists a constant $c>0$ such that $a_n\leq c b_n$ holds for all $n\in\NN$; and $a_n\asymp b_n$, when there exist constants $c_1, c_2>0$ such that $c_1a_n\leq b_n\leq c_2a_n$ holds for all $n\in\NN$. 
 \subsection{Standard definition of convergence rates}
 \label{sec:old_def}
Suppose we observe a random i.i.d.\, sample $X_1, \dots, 
X_n\in\mathcal{X}$ from a distribution $P_\theta$ indexed by a
parameter $\theta\in\Theta$, where $\Theta$ is potentially infinite
dimensional. Consider the problem of estimating parameter $\theta$
from the available data $X^n=(X_1, \dots, X_n)$. We measure the
estimation error with respect to some metric
$d:\Theta\times\Theta\to\RR^+_0$. In order to choose an estimator for a
particular setting it is important to have a way of comparing the
performance of estimators. The minimax paradigm offers a classic
solution for performance evaluation. It judges the performance of an
estimator by its rate of convergence, which is defined by
taking the worst case scenario over all elements in the given
parameter space. More precisely, we define an {\em estimator\/} $\hat\theta$ to be a collection $\{\hat\theta_n\}_{n \in \NN}$ such that for each $n$, $\hat\theta_n =\hat\theta(X^n):\mathcal{X}^n\to\Theta$ is a function from samples of size $n$ to $\Theta$. We say that
$\hat\theta$ has a {\it rate
 of convergence} $f_{\hat\theta}:\NN\to\RR^+$ if 
\begin{enumerate}[(i)]
\item there exists $C>0$ such that for every sample size $n\in\NN$
\[\sup_{\theta\in \Theta}\EE_{X^n \sim \PP_\theta}\left[\frac{d(\theta, \hat\theta_n)}{f_{\hat\theta}(n)}\right]\leq C. 
\]
\item For any function $\tilde f:\NN\to\RR^+$ such that $f_{\hat\theta}(n)/\tilde f(n)\to\infty$ 
\[
\sup_{\theta\in \Theta}\EE_{X^n \sim \PP_\theta}\left[\frac{d(\theta, \hat\theta_n)}{\tilde f(n)}\right]=\infty. 
\]
\end{enumerate}
An estimator $\hat\theta$ is called {\it minimax optimal} (up to a constant factor), if
\begin{equation}\label{eq:supper}
f_{\hat\theta}(n) \asymp\inf_{\tilde\theta} {f_{\tilde\theta}(n)}, 
\end{equation}
where the infimum is taken over all estimators that can be defined on the domain.

Here we expressed the minimax rate in terms of a supremum over a
ratio. It is perhaps more common to express $\hat\theta$ being minimax optimal, i.e. (\ref{eq:supper}), without
using ratios, but directly (yet equivalently) as
$$\sup_{\theta\in \Theta}\EE_{X^n \sim \PP_\theta}\left[{d(\theta, 
 \hat\theta_n)}\right]\leq C \inf_{\tilde\theta} \sup_{\theta\in
 \Theta}\EE_{X^n \sim \PP_\theta}\left[{d(\theta, 
 \tilde\theta_n)}\right].$$ Under this formulation, the
straightforward extension to taking a worst-case over time trivializes
the problem: if we take the supremum on the right not just over
$\theta \in \Theta$ but also over $n$, it will be achieved for
$n=1$ (or other small sample sizes) --- Nature would always choose the smallest possible sample size
and the problem would become uninteresting. By rephrasing minimax
optimality in terms of ratios, and taking a supremum over stopping
times/rules, we do get a useful extension, as we now show.

\subsection{Time-Robust Convergence rates}
\label{sec:new_def}
The classic definitions for minimax rates assume the sample size is
fixed and known in advance. Now we propose our generalized definitions
that account for not knowing the sample size in advance.
 
Let $\TT$ be a collection of all possible almost surely finite stopping times with respect to the sequence of sigma algebras $\FF_n=\sigma(X_1, \dots, X_n)$ generated by the data $X^n. $ We say that an estimator $\hat\theta$ (with
$\hat\theta_n=\hat\theta(X^n)$) has a {\it weakly adversarial time-robust rate of convergence\/} $f_{\hat\theta}:\NN\to\RR^+$ if 
\[
\sup_{\theta\in \Theta}\sup_{\tau\in\TT}\EE_{X^\infty \sim \PP_\theta}\left[\frac{d(\theta, \hat\theta_\tau)}{f_{\hat\theta}(\tau)}\right]\leq C
\]
and for any function $\tilde f:\NN\to\RR^+$ such that $f_{\hat\theta}(n)/\tilde f(n)\to\infty$ 
\[
\sup_{\theta\in \Theta}\sup_{\tau\in\TT}\EE_{X^\infty \sim \PP_\theta}\left[\frac{d(\theta, \hat\theta_\tau)}{\tilde f(\tau)}\right]=\infty. 
\] 
An estimator $\hat\theta$ is {\it weakly adversarial time-robust minimax optimal} if its weakly-adversarial time-robust rate of convergence $f_{\hat\theta}$ satisfies
\begin{equation}\label{eq:weaklymmo}
f_{\hat\theta}(n)\asymp \inf_{\tilde\theta}f_{\tilde\theta}(n), 
\end{equation}
where the infimum is taken over all estimators. Then the function
$f_{\hat\theta}$ is called the {\it weakly adversarial time-robust minimax rate} for the
given statistical problem.

We say that an estimator $\hat\theta$ (with
$\hat\theta_n=\hat\theta(X^n)$) has a {\it strongly adversarial time-robust rate of convergence\/} $g_{\hat\theta}:\NN\to\RR^+$ if 
\[
\sup_{\theta\in \Theta} \EE_{X^\infty \sim \PP_\theta} \left[ \sup_{n\in\NN} \frac{d(\theta, \hat\theta_n)}{g_{\hat\theta}(n)}\right] \leq C
\]
and for any function $\tilde g:\NN\to\RR^+$ such that $g_{\hat\theta}(n)/\tilde g(n)\to\infty$ 
\[
\sup_{\theta\in \Theta}\EE_{X^\infty \sim \PP_\theta}\left[\sup_{n\in\NN}\frac{d(\theta, \hat\theta_n)}{\tilde g(n)}\right]=\infty. 
\] 
An estimator $\hat\theta$ is {\it strongly adversarial time-robust minimax optimal} if its strongly adversarial time-robust rate of convergence $g_{\hat\theta}$ satisfies (\ref{eq:weaklymmo}) with $f_{\cdot}$ replaced by $g_{\cdot}$, 
where again the infimum is taken over all estimators. Then the function
$g_{\hat\theta}$ is called the {\it strongly adversarial time-robust minimax rate} for the
given statistical problem.

We may also call the weakly adversarial time-robust rate of convergence the {\em
 always-valid convergence rate}, since the freedom in when to stop is
exactly the same as in the recent papers on {\em always-valid\/} (also
known as `anytime-valid') confidence intervals and $p$-values. The
strongly adversarial time-robust rate may also be called the {\em
 worst-case-sample size\/} convergence rate. Statistical estimation
in which the stopping time $\tau$ is not known in advance is often
referred to as estimation with {\em optional stopping}. However, in
e.g. the Bayesian literature this is usually interpreted as `the
stopping rule may be unknown, but it is chosen independently of
$\theta$. We may think of the weakly time-robust or ``always-valid''
rate as the rate obtained in a setting with a stronger form of
optional stopping, in which Nature jointly chooses $\theta$ and the
stopping time, which can then be chosen as a function of
$\theta$. Note that choosing a stopping time is equivalent to choosing
a stopping {\em rule}, which, at each sample size $n$ decides, based
on $\theta$ and all past data, whether to stop or not. In contrast, 
the strongly adversarial time-robust rate corresponds to deciding to stop at the
worst $n$, a rule does not depend on the true $\theta$ but instead, 
unlike a stopping time, requires a look into the future.

Clearly any estimator $\hat\theta$, if it has strongly adversarial time-robust
rate $f(n)$, has weakly adversarial time-robust rate and standard minimax rate that are at most $f(n)$. Similarly, any standard
minimax rate can be no larger, up to a constant factor, than any
weakly adversarial time-robust minimax rate, which in turn can be no larger, up to
a constant factor, than the corresponding strongly adversarial time-robust minimax
rate. In the next section we study the relationship between these
three quantities more closely.
 
\section{Main results}
\label{sec:main}
Our first result gives a general upper bound on the strongly adversarial
time-robust minimax rate, and hence also on the weakly adversarial time-robust
minimax rate, as compared to the usual minimax rate. This result holds
under very weak conditions for any parameter estimation setting.
Then, we consider estimating the mean parameter in the Gaussian
location family with the usual Euclidean distance.
It turns out that for this problem both the weakly and the
strongly adversarial time-robust minimax rates are equal to $n^{-1}\log\log n$, 
while the usual minimax rate is $n^{-1}$.

\subsection{Time-robust rates are never much worse than standard rates}
In the theorem below we show that the strongly (and hence the weakly)
adversarial time-robust minimax rate differs from the usual minimax rate by at
most a logarithmic factor under a very mild assumption on the decay of
the usual minimax rate function. The result makes no assumption about the metric $d$. 
\begin{thm}
\label{general}
Let $f:\NN\to\RR^+$ be a minimax rate for some given statistical
estimation problem, such that $f$ is non-increasing and
\begin{equation}
\label{eq:weak}
\frac{f(2n)}{f(n)}\geq C
\end{equation}
for some $C>0$. Then the strongly adversarial time-robust minimax rate $g(n)$ for
the same problem satisfies
\[
g(n)\lle f(n)\log n. 
\]
\end{thm}
Notice that the assumption \eqref{eq:weak} holds for $f(n) \asymp n^{-\gamma} (\log n)^{\beta}$ with $0 < \gamma \leq 1$, $\beta \geq 0$, which is equal to the minimax rate for most standard parametric and nonparametric estimation problems, and under most standard metrics, see e.g. \cite{tsybakov}. The proof of Theorem \ref{general} involves constructing an estimator
that uses only part of the available data. We then show that the
standard minimax rate for this estimator is $f(n)\log n$ and that it remains unaffected if we include the supremum over $n$.

\subsection{The time-robust rate can be different from the standard rate}
Now we present a problem for which the time-robust minimax rates, 
while equal to each other, do not coincide with the usual minimax
rate. Consider a Gaussian location family with fixed variance $\{
\PP_\mu, \mu \in \RR\}$, where each $\PP_\mu$ is a Gaussian distribution with mean $\mu$ and variance one. Let
$d(\mu, \mu')=(\mu-\mu')^2$ be the usual Euclidian distance.
The following
theorem shows that the strongly adversarial time-robust minimax rate for
estimating $\mu$ is upper bounded by $n^{-1}{\log\log n}$ and the
weakly adversarial time-robust minimax rate is lower bounded $n^{-1}{\log\log n}$, 
so that both rates coincide and are equal to $n^{-1}{\log\log
 n}$. Furthermore, it shows that the rate is attained by the maximum
likelihood estimator (MLE).
 
 To avoid taking a logarithm of a negative number we set $f(n)=1$ for $n=1, 2$ and
 \[
 f(n)=n^{-1}{\log\log n}, \text{ when } n\geq3. 
 \]
 \begin{thm} 
 \label{thm:exp}
Let $\{\PP_\mu, \mu\in\RR\}$ represent the Gaussian location family,
i.e. under $\PP_\mu$, the $X_1,X_2, \ldots$ are i.i.d. $\sim \NNN(\mu,
1)$. Then,
 \begin{enumerate}[(i)]
\item {\it Upper bound. } There exists a constant $C>0$ such that 
\begin{equation}
\label{tough_part}
\sup_{\mu \in \RR} \EE_{X^{\infty}\sim \PP_\mu} \left[\sup_{n\in\NN} \frac{ (\mu- \MMLE_n)^2}{f(n)} \right]\leq C, 
\end{equation}
where $\MMLE_n=\MMLE(X^n)$ is the maximum likelihood estimator. 
\vspace{0.4cm}
\item {\it Lower bound. } Let $\TT$ be the collection of all (a.s.\, finite) stopping times w.r.t.\, the filtration $\FF=\{\sigma(X^n), n\in\NN\}$. 
 There is a $C > 0$ such that for
any estimator $\hat\mu$ with $\hat\mu_n=\hat\mu(X^n)$, 
\begin{equation}
\label{eq:friday0}
\sup_{\mu\in \RR}\sup_{\tau\in\TT}\EE_{X^\infty\sim\PP_\mu}\left[\frac{(\mu-\hat\mu_\tau)^2}{f(\tau)}\right]\geq C, 
\end{equation}
and 
\begin{equation}
\label{eq:friday}
\sup_{\mu\in \RR}\sup_{\tau\in\TT}\EE_{X^\infty\sim\PP_\mu}\left[\frac{(\mu-\hat\mu_\tau)^2}{g(\tau)}\right]=\infty, 
\end{equation}
for all non-increasing $g:\NN\to\RR^+$ such that $f(n)/g(n)\to\infty$. 
\end{enumerate}
 \end{thm}
\begin{rem}
The upper bound of Theorem \ref{thm:exp} holds more
 generally for all exponential families which satisfy a mild condition
 that holds for most families used in practice --- we give the
 generalized upper bound in Section \ref{sec:proofs}. We conjecture
 that the lower bound can also be generalized to all exponential
 families, although due to the rather involved proofs we leave
 the generalization outside of the scope of this paper.
 \end{rem}
The upper bound is, unsurprisingly, based on the law of iterated
logarithm. To prove the result we use the techniques developed for
iterated logarithm martingale concentration inequalities, see
e.g. \cite{darling, balsubramani, howard}, and combine them with the
fundamental results from \cite{shafer2011test} to connect test
martingales to $p$-values and so-called $E$-values.

The proof for the lower bound relies on a number of steps.  We first
show that the bound must hold for estimators that are `MLE-like': they
are sufficiently `close', in a particular sense, to the MLE and hence
also to Bayes optimal estimators based on standard priors. In the
second step, we relate the problem of bounding the minimax risk to the
problem of bounding the Bayes risk. This idea is not new in itself and
is widely used in minimax theory, see e.g. \cite{tsybakov}. The
difficulty we face is that this standard argument does not give
anything useful if directly applied to MLE-like (and standard
Bayes-like) estimators.  In the final step of the proof, we thus
construct a stopping rule for each `non-MLE-like' estimator that stops
when the estimator is far away from the MLE, and give a lower bound
for the Bayes risk, i.e. the risk of the non-MLE-like estimator under
the Bayesian posterior. The complete proofs are given in Section
\ref{sec:proofs}.
 \subsection{Application: avoiding the AIC-BIC Dilemma in Model Selection and Post-Selection Inference}
 \label{sec:application}
Consider a simple model selection problem. Data $X^n$ are used to
select between two nested exponential family models, with
$\MMM_0=\{p_{\mu_0}, \mu_0 \in M_0 \}$ and $\MMM_1=\{p_\mu, \mu\in M\}$, 
where $M_0 \subset M$, and $\MMM_1$ is an exponential family with
mean-value parameter set $M\subset\RR^k$. 
For simplicity let $M_0 = \{ \mu_0 \}$, so that $\MMM_0$ is a
singleton. Examples include testing whether a coin is fair (using
Bernoulli distributions) or whether a treatment has an effect (using
the Gaussian location family). We consider combined model selection
and estimation: first, a model is selected using some model selection
method such as for example, AIC, BIC, cross-validation, Bayes factor
model selection, or one of its many variations. Then, based on the
chosen model, a parameter within the model is estimated using some
estimator such as the MLE or a Bayes predictive distribution (note
that if the singleton model $\MMM_0$ is selected, then the estimator
must return $\mu_0$).

Two desirable properties for such combined procedures are (i)
consistency and (ii) rate optimality of the post-model selection
estimation. \citet{yang} shows that at least for some settings having
both (i) and (ii) at the same time is impossible: any combination of a
consistent model selection and subsequent estimation method misses the
standard minimax rate (equal to $n^{-1}$ in our case) by a factor
$g(n)$ such that $g(n)\to\infty$, as $n\to\infty$. Yang's setting can
be adjusted to include the exponential family setting presented here, 
see \cite{almost} for more details. Yang also shows that a similar
problem occurs if we average over the models (in a Bayesian or any
other way) rather than select one of the models. This has been called
the {\em AIC-BIC dilemma}. While for the mathematician, it may not be
so surprising that there is no procedure which is optimal under two
different definitions of optimality, it has been argued (by Yang and
many others) that for the practitioner, there really is a dilemma: she
simply wants to get an initial idea of which model best explains her
data, indicating how to focus her subsequent research, and views
consistency and rate optimality as desirable properties, both
indicating that her procedure will do something reasonable in
idealized situations, but neither one being the ultimate goal. Which
of the two properties is more important is then often not clear.

However, if one accepts the novel definition of minimax rate presented
here, one has a way out of the dilemma after all: there exist model
selection procedures that are strongly consistent, while, if combined
with the MLE, have a standard convergence rate equal to $n^{-1}\log\log
n$ under the squared error loss. \cite{almost} showed this explicitly
for model selection based on the {\em switch distribution\/}
introduced by \cite{tim}. The switch distribution was
 specifically designed for this purpose, but some other methods
 achieve this as well. For example, while Bayesian model selection
 based on standard priors achieves only an $n^{-1} \log n$ rate
 \citep{tim}, it seems quite likely that if $M_1$ is equipped with the
 quite special {\em stitching priors\/} \citep{howard} which
 asymptote at $\mu_0$, one can also get strongly consistent model
 selection and an $n^{-1}\log\log n$ estimation rate by Bayes factor
 model selection. Since we have no explicit proof of this, we
 continue the discussion with switching rather than stitching.
The estimation rate for
the switch distribution (at least for the Gaussian location family, but we
conjecture for general exponential families) is equal to the time-robust minimax
convergence rate, derived in Theorem
\ref{thm:exp}. Thus the switch procedure is both strongly consistent {\em
 and\/} minimax optimal in the new, time-robust, sense. We see that
using time-robust definitions of minimax optimality, the gap between
(i) and (ii) above can be bridged, whereas, by Yang's theorem, this is
impossible under the standard definition of minimax rate. Hence, by
redefining minimax optimality so as to be robust with respect to {\em
 all\/} parameters (including $n$) that we as statisticians do not
have under control, the minimax rate slightly changes and the {\em
 AIC-BIC dilemma\/} simply disappears: combined consistency and
estimation optimality can be achieved by, for example, the switch
distribution. As an aside, neither AIC nor BIC itself `solve' the
dilemma with the time-robust definitions: AIC is still inconsistent, 
whereas BIC, when combined with efficient post-selection estimation, 
achieves a standard estimation rate of order $n^{-1} \log n$; since
the time-robust rate is at least the standard rate, it must still be
rate-sub-optimal under the time-robust definition of minimax rate.

\section{Discussion}
 \label{sec:discussion}
In this paper we suggested a generalization of minimax theory enabling
it to deal with unknown and data-dependent sample sizes. We introduced
two notions of time-robust minimax rates and compared them to the
standard notion of minimax rates. We showed that for most problems the
rates differ by at most logarithmic factor. We also provided an
example of a (parametric) setting, for which the weak and the strong
rates are the same, yet they differ from the standard rates by an
iterated logarithmic factor. However, it is not yet clear under what
circumstances the logarithmic upper bound on the difference, derived
in Theorem \ref{general}, is tight: for example, it might be possible
that in some standard (e.g. nonparametric) problems, the gap vanishes
(the strongly adversarial time-robust rate is within a constant factor of the
standard rate); but in others it may even be larger than order $\log
\log n$. Similarly, in some settings, the weak and strong time-robust
rates may coincide, and in some others they may differ. A major goal
for future research is thus to sort out more generally when the three
rates coincide and when they differ, and if so, by how much.

\section{Proofs} 
\label{sec:proofs}
\subsection{Proof of Theorem \ref{general}}

Let $\hat\theta$ be an estimator that achieves the standard minimax rate, i.e.\ there exists $C'>0$ such that for every $n\in\NN$
\[
\sup_{\theta \in \Theta} \EE_{X^n \sim P_{\theta}} \left[\frac{d(\hat\theta_n, \theta)}{f(n)}\right] \leq C', 
\] 
where $\hat\theta_n=\hat\theta{(X^n)}$. For $k \geq 0$ define 
\[
\hat\theta_{(k)} = \hat\theta(X_1, \ldots, X_{2^k}). 
\] 
Let $\lfloor x\rfloor=\max\{z\in\ZZ, z\leq x\}$. Consider 
\[
\hat\theta'_n= \hat\theta'{(X^n)}=\hat\theta_{(\lfloor \log_2 n \rfloor)}. 
\] to be the function of
data $X^n$ that outputs
$\hat\theta_{(\lfloor \log_2 n \rfloor)}$. In what follows we show that $\hat\theta'$ achieves a time-robust minimax rate satisfying the claim of the theorem. 

Consider a probability mass function $\pi:\NN_0\to[0, 1]$ with $\sum_{j\geq 0} \pi(j) = 1$. Denote $\EE_\theta[\cdot]=\EE_{X^\infty\sim\PP_\theta}[\cdot]$. Because of assumption \eqref{eq:weak} for any $\theta\in\Theta$ we have 
\begin{align*}
\EE_{\theta}\left[ \sup_{n\in\NN} \pi(\lfloor \log_2 n \rfloor) \cdot 
 \frac{d(\hat\theta'_n, \theta)}{f(n)} \right] &\leq C \EE_{\theta}\left[ \sup_{j\in\NN} \pi(j) \cdot 
 \frac{d(\hat\theta_{(j)}, \theta)}{f(2^j)} \right]
\leq \\
 &\leq C \EE_{\theta}\left[ \sum_{j\in\NN} \pi(j) \cdot 
 \frac{d(\hat\theta_{(j)}, \theta)}{f(2^j)} \right]\leq C \sup_{j\in\NN} \EE_{\theta}\left[
 \frac{d(\hat\theta_{(j)}, \theta)}{f(2^j)} \right] . 
\end{align*}
The last inequality is due to $\sum_{j\geq 0} \pi(j) = 1$. Since $\hat\theta$ achieves the standard minimax rate we have
\begin{align*}
& \sup_{j\in\NN} \EE_{\theta}\left[
 \frac{d(\hat\theta_{(j)}, \theta)}{f(2^j)} \right] 
 \leq \sup_{n\in\NN} \EE_{\theta}\left[
 \frac{d(\hat\theta_{n}, \theta)}{f(n)}\right] \leq C'. 
\end{align*}
Putting everything together we arrive at
\[
\EE_{\theta}\left[\sup_{n\in\NN} \pi(\lfloor \log_2 n \rfloor) \cdot 
 \frac{d(\hat\theta'_ n, \theta)}{f( n)} \right]\leq CC'. 
\]
Let $\pi(j) \asymp j^{-1 - \alpha}$. Then for every $\alpha > 0$ there exists a constant $C''>0$ such that for any $\theta\in\Theta$ 
$$
\EE_{\theta}\left[\sup_{n\in\NN} \frac{1}{(\lfloor \log_2 n \rfloor + 1)^{1+ \alpha}} \cdot 
 \frac{d(\hat\theta'_ n, \theta)}{f( n)} \right] \leq C''. 
$$
This finishes the proof of the theorem.

\subsection{ Proof of {\it(i)} in Theorem \ref{thm:exp} }

We will prove a more general result that holds for all exponential families.

\subsubsection{Preliminaries}
Consider an exponential family $P_{\bar\Theta}=\{ \PP_\theta, \theta \in \bar \Theta\}$, $\bar\Theta \subset \RR^k$, defined as the family of densities
\[
p_\theta(x)=r(x)e^{\theta^{T} \phi(x)-\psi(\theta)}, \, x\in\XX, \, \theta\in\bar\Theta, 
\]
Here $\phi(x)$ is a sufficient statistics for $\theta$.  When we write
$X^{\infty} \sim \PP_{\theta}$ we mean that $X_1, X_2, \ldots$ are
i.i.d. with each $X_i \sim \PP_{\theta}$.  We use the mean-value
parametrization of the exponential family and set $P_{\bar M} =\{
\PP_\mu, \mu \in \bar M\}$, $\bar M\subset\RR^k$ with the link
function
\begin{equation}\label{eq:helpthereader}
\mu(\theta)=\EE_{X\sim\PP_\theta}\left[\phi(X)\right]. 
\end{equation}
We let $\theta(\cdot)$, the inverse of $\mu(\cdot)$, be the transformation function from the mean-value parametrization to the canonical one. $\theta(\cdot)$ exists for all exponential families, see e.g. \cite{brown}.

We assume the parameter space $\bar M$ is such that maximum likelihood
estimator lies in $\bar M$ and is unique. That means we potentially
have to extend the original family $\{ \PP_{\theta}, \theta \in
\bar\Theta\}$ to accommodate that by including distributions `on the
boundary'. For example, in the Bernoulli model, the natural parameter
ranges from $- \infty$ to $\infty$, corresponding to $\{\PP_\mu, \mu
\in (0, 1)\}$, excluding the degenerate distributions $\PP_0$ and
 $\PP_1$. We then simply set $\bar{M} = [0, 1]$ to include these distributions.
 
More formally, the assumption is as follows. 
\begin{assumption}[]
 \label{mle_form}
 $\bar M$ is such that the maximum likelihood estimator
 $\MMLE_n=\MMLE(x^n)$ satisfies
 \[
 \MMLE_n=\frac1n\sum_{i=1}^n \phi(x_i)\in \bar M
 \]
 for all $x_1, \dots, x_n\in\XX$. 
 \end{assumption}
This assumption is needed since we are using the properties of the average of i.i.d.\, random variables to prove a statement about the MLE. However, the assumption is rather weak: most standard exponential families either satisfy
Assumption~\ref{mle_form} or can be extended to satisfy it; see
Chapter 5 of \cite{brown}.

 Furthermore, we introduce a definition of a CINECSI subset of $\bar M$. 
 \begin{df}
 A CINECSI (Connected, Interior-Non-Empty Compact Subset of Interior) subset of a set $ \bar M$ is a connected subset of the interior of $\bar M$ that is itself compact and has nonempty interior. 
 \end{df}
For discussion on CINECSI subsets see \cite{peterbook}. 
 
Finally, we introduce an additional assumption on the set of true parameters $M\subseteq \bar M$, from which the data is assumed to be generated and over which we are taking the supremum.

\begin{assumption}[]
 \label{inside}
$M\subseteq \bar M$ is such that there exist constants $\sigma>0$ and $\delta>0$ such that for all $\eta\in \RR^k$ with $\|\eta\|^2\leq\delta$, where $\|\cdot\|$ is the usual Euclidian distance, and all $\mu\in M$
\begin{equation}
\label{asd}
\EE_{X\sim \PP_\mu}\left[{e^{\eta^T(\phi(X)-\mu)}}\right]\leq e^{\sigma\eta^T\eta/2}. 
\end{equation}
 \end{assumption}
 
 In the proposition below we show that the Assumption \ref{inside} is satisfied for the Gaussian location family with $M=\bar M$ and for other exponential families when $M$ is a CINECSI subset of $\bar M$. This condition is required in our proofs for bounding Fisher information, but might potentially be relaxed if one uses different proof techniques. 

\begin{prop}Assumption \ref{inside} is satisfied for the following settings:

\begin{enumerate}[(i)]
\item When $P_{\bar M}=\{\PP_\mu, \mu\in\bar M\}$ is Gaussian location family, i.e.\ $\PP_\mu$ represents a $\NNN(\mu, 1)$ distribution, and $M=\bar M=\RR.$
\item When $P_{\bar M}$ is any exponential family and $M$ is a CINECSI subset of $\bar M$.
\end{enumerate}
\end{prop}
\begin{proof} 
$\phantom{x}$

\begin{enumerate}[(i)]
\item For the Gaussian location family $\phi(X)=X$. Inspecting the definition of the moment generating function, we immediately find that  for all $\eta\in\RR$
\[
\EE_{X\sim \PP_\mu}\left[{e^{\eta(X-\mu)}}\right] = e^{\eta^2/2}.
\]
Then \eqref{asd} is satisfied with $\sigma=1$ and any $\delta>0.$
\item Let $M$ be a CINECSI subset of $\bar M$. Consider the canonical parametrization of the exponential family with $\Theta=\theta(M)$ (where $\theta(\cdot)$ is as defined underneath (\ref{eq:helpthereader}). By Taylor expansion we have for every $\eta \in\RR^k$ and any $\theta\in\Theta$
\[
\EE_{X\sim \PP_{\theta}}\left[{e^{\eta^T\phi(X)}}\right]=e^{\psi(\theta+\eta)-\psi(\theta)}=e^{\eta^T\mu(\theta)+\eta^T I(\theta')\eta/2}, 
\]
where $I(\cdot)$ is Fisher information and $\theta'$ is between $\theta$ and $\theta+\eta$. Now we construct a set $B_\delta(0)=\{\eta\in \RR^k:\|\eta\|^2\leq\delta\}$ such that for all $\eta\in B_\delta(0)$ the Fisher information at $\theta'$ (located between $\theta$ and $\theta+\eta$) is bounded.

Notice that since $\mu(\cdot)=\theta^{-1}(\cdot)$ is continuous, $\Theta$ is a CINECSI subset of $\bar\Theta=\theta(\bar M)$. Hence, there exists $\delta>0$ and $\Theta^\delta$ such that $\Theta\subset \Theta^\delta$, $\Theta^\delta$ is a CINECSI subset of $\bar \Theta$, and
\[
\inf_{\theta\in\Theta, \theta'\in\partial \Theta^\delta}\|\theta-\theta'\|\geq\delta. 
\]
 Then for all $\eta\in B_\delta(0)$, we have $\theta'\in \Theta^\delta$. Since $\Theta^\delta$ is a CINECSI subset of $\bar \Theta$, the Fisher information is bounded on $\Theta^\delta$. Therefore, there exists $\sigma=\sup_{\theta'\in \Theta^\delta} I(\theta')>0$ such that
\[
\EE_{X\sim \PP_\mu}\br{e^{\eta^T(\phi(X)-\mu)}}\leq e^{\sigma\eta^T\eta/2}
\]
for all $\eta\in B_\delta(0)$ and all $\mu\in M$.
\end{enumerate}
\end{proof}

\subsubsection{General theorem}
In the following theorem we show that under Assumptions \ref{mle_form}
and \ref{inside} the strongly adversarial time-robust minimax rate for
the MLE is at most $n^{-1}\log\log n$. Note that below, the MLE
$\MMLE$ is defined relative to the full set $\bar{M}$, not the
potentially restricted set $M$. Also, observe that \eqref{tough_part}
directly follows from Theorem \ref{thm:upperbound}, since for the
Gaussian location family, Assumption \ref{mle_form} is satisfied for
$\bar M=\RR$.
\begin{thm}\label{thm:upperbound}
Let $\bar M$ be such that Assumption \ref{mle_form} is satisfied. Let $M\subseteq\bar M$ be such that Assumption \ref{inside} is satisfied. Then there exists a constant $C>0$ such that 
\begin{equation*}
\sup_{\mu \in M} \EE_{X^{\infty}\sim \PP_\mu}\left[ \sup_{n \in \NN} \frac{\|\mu- \MMLE_n \|^2}{f(n)}\right]\leq C, 
\end{equation*}
where $f(n)=n^{-1}\log\log n$ for $n\geq 3$ and $f(n)=1$ for $n=1, 2.$
\end{thm}

\subsubsection{Proof of Theorem \ref{thm:upperbound}}

Let $
S_n=\sum_{i=1}^n (\phi(X_i)-\mu(\theta))$. Notice that for $n\geq 3$ we have $n\log\log n\geq 1/4.$ Furthermore, there exists a constant $C'$ such that 
\[
\sum_{m=1}^{27} \|S_m\|^2\leq C'.
\]
Also, for every $\mu\in M$
\[
\EE_{X^{\infty}\sim \PP_\mu}\left[\sup_{n \in \NN} \frac{\|\mu- \MMLE_n \|^2}{f(n)}\right]\leq \EE_{X^{\infty}\sim \PP_{\mu}}\left[4\sum_{m=1}^{27} \|S_m\|^2+\sup_{n>27} \frac{\|S_{n}\|^2}{{n \log\log n }}\right].
\]
It is then sufficient to show that there exists a constant $C>0$ such that for every parameter $\theta \in \Theta=\theta(M)$ (where $\theta(\cdot)$ is as defined underneath (\ref{eq:helpthereader})),
\[
\EE_{X^{\infty}\sim \PP_\theta} \left[\sup_{n> 27} \frac{\|S_{n}\|^2}{{n \log\log n }}\right]\leq C.
\]
First, following \cite{shafer2011test}, we define a {\em test supermartingale\/}
$(U_n)_{n \in \NN}$ relative to filtration $( \FF_n)_{n \in \NN}$ and distribution $\PP$ to
be a nonnegative supermartingale relative to $( \FF_n )_{n \in \NN}$
with starting value bounded by $1$, i.e. $(U_n)_{n \in \NN}$ is a
test martingale iff for all $n \in \NN$, $U_n \geq 0$ a.s., $\EE_{\PP} [U_n
 \mid \FF_{n-1}] \leq U_{n-1}$, and
$\EE[U_1] \leq 1$. 
The following lemma is an immediate consequence of combining two of
\cite{shafer2011test}'s fundamental results: 
\begin{lem}
\label{shafer}
Suppose that $(U_n)_{n \in \NN}$ is a test supermartingale under
distribution $\PP$. Then
\[
\EE_{\PP} \br{\sup_{n\in\NN}\sqrt{U_n}/2} \leq 1.
\]
\end{lem}
\begin{proof}
Let $V = (1/(\sup_{n \in \NN} U_n)$. From Theorem 2, part (1) of
 \cite{shafer2011test} we have that, for all $0 \leq \alpha \leq 1$, 
 $P(V \leq \alpha) \leq \alpha$, i.e. (the value taken by) $V$ can be
 interpreted as a $p$-value. Now Theorem 3, part (1) of
 \cite{shafer2011test}, together with (8) in that paper instantiated
 to $\alpha = 1/2$, gives that $1/(2 \sqrt{V})$ is an {\em
 $E$-variable}, i.e. $\EE [1/(2 \sqrt{V})] \leq 1$, and the first result
above follows. [In the terminology of \cite{shafer2011test}, a random
 variable $W$ with $\EE[1/W] \leq 1$ is called ``Bayes factor''. In
 recent publications, the terminology has changed to calling $1/W$ an
 $E$-variable and its value $E$-value
 \citep{vovk2019combining, GrunwaldHK19}.]
\end{proof}
Consider $\delta>0$ and $\sigma>0$ such that Assumption \ref{inside} is satisfied. Let 
\[
B_\delta(0)=\{\eta\in \RR^k:\|\eta\|^2\leq\delta\}. 
\]
For a probability distribution on $B_\delta(0)$ with the density function $\gamma:B_\delta(0)\to\RR$ define 
\[
Z_n=\int_{B_\delta(0)}\gamma(\eta)e^{\eta^T S_n-n\sigma\eta^T\eta/2}d\eta.
\]
Additionally, let $Z_0=1.$ Due to the properties of the conditional expectation and Assumption \ref{inside} we know that $(Z_n)_{n \in \NN}$ is a test
supermartingale relative to filtration $(\sigma(X^n))_{n\in\NN}$. Then Lemma \ref{shafer} has the following corollary.
\begin{cor}
For any distribution $\gamma$ on $B_\delta(0)$
\label{mart}
\begin{equation}\label{eq:yes}
\EE_{X^{\infty} \sim \PP_{\theta}} \left[\sup_{n\in\NN}\sqrt{Z_{n}}\right] \leq 2. 
\end{equation} 
\end{cor}

By choosing the right distribution on $\eta$, i.e. the right $\gamma$, we can show that (\ref{eq:yes}) implies the following lemma (we provide the proof for this lemma in the next subsection).

\begin{lem}
\label{main}
Let $S_n=(S_n^1, \dots, S_n^k)^T$ and $T_n=|S_n^1|+\dots+|S_n^k|$. For every $c < \frac{2\sqrt 2\delta}{k}$ and
for all $\theta\in \Theta$, 
\[
\EE_{X^\infty\sim\PP_\theta} \left[\sup_{n>27}e^{c\frac{T_{n}}{\sqrt{n\log\log{n}}}}\mathbbm{1}_{A_c}\right] \leq e^{K_1+K_2 c^2}, 
\]
where $K_1= 1.5+(k+1)\log2$, $K_2=18\sigma k$, and $A_c=\left\{\sup_{n>27}\frac{T_n}{\sqrt{n\log\log n}}\geq \frac{1}{2c}\left(2K_2c^2+3\right)\right\}$.
\end{lem}
By Markov's inequality for any $a>0$ and any $c>0$
\begin{multline*}
\PP\left[\ind{A_c}\sup_{n>27}\frac{T_{n}^2}{{n\log\log{n}}}\geq a\right]=\PP\left[\ind{A_c}\sup_{n>27}e^{{c\frac{T_{n}}{\sqrt{n\log\log{n}}}}}\geq e^{c\sqrt a}\right]\leq\\ 
\leq e^{-c\sqrt a}\EE\left[\ind{A_c}\sup_{n>27}e^{{c\frac{T_{n}}{\sqrt{n\log\log{n}}}}}\right]. 
\end{multline*}
Combining it with Lemma 5.2 we get that for all $c < \frac{2\sqrt 2\delta}{k}$
\begin{multline*}
\EE\left[\ind{A_c}\sup_{n>27}\frac{T_{n}^2}{{n\log\log{n}}}\right]=\int_{0}^\infty \PP\left[\ind{A_c}\sup_{n>27}\frac{T_{n}^2}{{n\log\log{n}}}\geq a\right] da \leq \\
\leq \int_0^\infty e^{-c\sqrt a+K_1+K_2c^2}da=2c^{-2}e^{K_1+K_2 c^2}. 
\end{multline*}
The minimum of the RHS is achieved, when $c=\min\left\{ \frac{2\sqrt 2\delta}{k}, \frac1{\sqrt {K_2}}\right\}.$
Also, 
\[
\|S_n\|^2=(S_n^1)^2+\dots(S_n^k)^2\leq (|S_n^1|+\dots+|S_n^k|)^2=T_n^2. 
\]
Therefore, for any $\theta\in \Theta$ and for $c=\min\left\{ \frac{2\sqrt 2\delta}{k}, \frac1{\sqrt {K_2}}\right\}$ we have
\begin{multline*}
\EE_{X^{\infty}\sim \PP_\theta} \left[\sup_{n> 27} \frac{\|S_{n}\|^2}{{n \log\log n }}\right] \leq
\EE_{X^{\infty}\sim \PP_\theta} \left[\ind{A_c}\sup_{n> 27} \frac{T_n^2}{{n \log\log n }}\right]+\\+
\EE\left[\ind{\left\{\sup_{n>27}\frac{T_n}{\sqrt{n\log\log n}}< \frac{1}{2c}\left(2K_2c^2+3\right)\right\}}\sup_{n>27}\frac{T_{n}^2}{{n\log\log{n}}}\right]\leq\\
\leq 2c^{-2}e^{K_1+K_2 c^2}+c^{-1}\left(2K_2c^2+3\right)/2 .
\end{multline*}
This finishes the proof of the theorem.

\subsubsection{Proof of Lemma \ref{main}}

For simplicity of exposition we only provide the proof for $k=1$, the proof for $k>1$ can be found in the \hyperref[appn]{Appendix}. 

Consider a discrete probability measure on $B_\delta(0)$ with density
\[
\gamma(\eta)=\sum_{i\in\NN}\gamma_i\ind{\eta=\eta_i}, 
\]
where $\gamma_i=\frac1{i(i+1)}$ and $\eta_i=c_0\sqrt{\frac{-\log\gamma_i}{e^i}}$ for a constant $c_0>0$ such that $\eta_i\in B_\delta(0)$ for all $i\in\NN$. Notice that the above holds for all $c_0<\delta$.
Then for any fixed $n>27$
\[
Z_n=\sum_{i=1}^{\infty}\gamma_ie^{\eta_i S_n-n\sigma\eta_i^2/2}\geq \max_{i\in\NN} \gamma_i e^{\eta_i S_n-n\sigma\eta_i^2/2}. 
\]
Let $i_0=\lfloor \log n\rfloor$, where
$
\lfloor x\rfloor=\max\{m\in\NN: m\leq x\}. 
$
Then we have
\begin{multline}
\label{tech}
Z_{ n } \geq \gamma_{i_0}e^{\eta_{i_0}S_ n - n \sigma\eta^2_{i_0}/2}\geq\\
e^{-\log\left(\llogn(\llogn+1)\right)} \cdot e^{-\frac{ n \sigma c_0^2\log\left(\llogn(\llogn+1)\right)}{2e^{\llogn}}} 
\cdot e^{c_0\sqrt{\frac{\log\left(\llogn(\llogn+1)\right)}{e^{\llogn}}}S_ n }\mathbbm{1}_{S_ n \geq0}. 
\end{multline}
Note that for $n\geq 3$, we have 
$
\log n\leq2\log(n-1)$. 
Also $\log n >3$, when $ n >27$. Therefore, 
\begin{equation}\label{eq:notleftmosta}
\log\left(\llogn(\llogn+1)\right)\geq \log(\log n -1)+\log\log n \geq 3/2\log\log n . 
\end{equation}
On the other hand, 
\begin{equation}\label{eq:leftmost}
\log\left(\llogn(\llogn+1)\right)\leq \log\log n +2\log\log n =3\log\log n . 
\end{equation}
Additionally, 
\begin{equation}\label{eq:notleftmostb}
 n /3\leq e^{\log n -1}\leq e^{\llogn}\leq e^{\log n +1}\leq 3 n . 
\end{equation}
We use (\ref{eq:leftmost}) to rewrite the first factor in (\ref{tech});
(\ref{eq:notleftmosta}) and (\ref{eq:notleftmostb}) to rewrite the second and third factor in (\ref{tech}). Then for any $n>27$
$$
Z_{ n } \geq e^{-3\log\log n -(9\sigma c_0^2/2)\log\log n +c_0\sqrt{\frac{\log\log n }{2 n }}S_ n }\mathbbm{1}_{S_ n \geq0}
= e^{u( n , S_{ n })} \cdot \mathbbm{1}_{S_ n \geq0}. 
$$
where
$$u( n , s) = \frac{c_0\sqrt2}{2}\log\log n \left(\frac{s}{\sqrt{ n \log\log n }}- \frac{\sqrt{2}}{2c_0}\left(9\sigma c_0^2+6\right) \right).
$$
Similarly, by defining $Z'_{n}$ just as $Z_n$ but now relative to a distribution with the same $\gamma_i$ but now $\eta_i=-c_0\sqrt{\frac{-\log\gamma_i}{e^i}}$ (rather than $\eta_i=c_0\sqrt{\frac{-\log\gamma_i}{e^i}}$ as before) for any $n>27$ we get
$
Z'_{ n } \geq e^{u( n , -S_{ n })} \cdot \mathbbm{1}_{S_ n <0}
$.
This equation and the previous one can be re-expressed as:
$$
\sqrt{Z_{ n}} \geq e^{u( n, S_{ n})/2} \cdot \mathbbm{1}_{S_ n\geq0}
\text{\ \ \ and\ \ \ }
\sqrt{Z'_{ n}} \geq e^{u( n, -S_{ n})/2} \cdot \mathbbm{1}_{S_n<0.}
$$
Since $T_n=|S_n|$, we combine the previous two equations with taking the supremum to get: 
\begin{multline*}
\sup_{n>27}e^{u(n, T_{ n})/2}\leq
 \sup_{n>27} e^{u( n, S_{ n})/2} \cdot \mathbbm{1}_{S_ n\geq 0}
 + \sup_{n>27} e^{u( n, - S_{ n})/2} \cdot \mathbbm{1}_{ S_ n < 0} \leq \\
 \leq\sup_{n>27}\sqrt{Z_{ n}} + \sup_{n>27}\sqrt{Z'_{ n}}.
 \end{multline*}
We can now invoke Corollary \ref{mart}, which gives us
$$
\EE\left[ \sup_{n>27} e^{u(n, T_{n})/2}\right] \leq {2} + {2} = 4.
$$
Observe that on the set $A_{c_0}=\left\{\sup_{n>27}\frac{T_n}{\sqrt{n\log\log n}}\geq \frac{\sqrt 2}{2c_0}\left(9\sigma k c_0^2+6\right)\right\}$ the expression inside the brackets in $u(n, T_n)$ is positive. Since $\log\log n>1$, when $n>27$, by writing out $u(\cdot, \cdot)$ in full we get:
\[
\EE\left[\sup_{n>27}
e^{\frac{c_0\sqrt2}{4}\frac{T_n}{\sqrt{n\log\log n}}-\frac{1}{4}
 \cdot \left(9\sigma c_0^2+6\right)}\mathbbm{1}_{A_{c_0}} \right]\leq 4.
\]
The desired result follows by taking $c=c_0/\sqrt{8}$. 
\smallskip

\subsection{Proof of {\it(ii)} in Theorem \ref{thm:exp}}
To prove the lower bound we first show in Lemma \ref{lem:lil} that by the law of iterated logarithm there is a constant $c>0$ such the distance between the (slightly modified) MLE and the truth is at least $c\cdot f(n)$ infinitely often. Then, in Lemma \ref{lem:twocases} we show that for each estimator either (\ref{eq:friday0}) and (\ref{eq:friday}) holds with some constant dependent on $c$ and the probability of being `close' to the MLE; or the estimator is `far' from the MLE infinitely often. Therefore, we only need to consider the latter estimators. For those estimators in Section \ref{br} we lower bound the minimax risk by the Bayes risk (a standard trick in minimax theory). Finally, in Section \ref{sec:stop} for each estimator we introduce a suitable stopping time such that the Bayes risk is bounded below by a constant that again depends on $c$ introduced above.

 \subsubsection{Reducing the number of considered estimators}
 
Consider a standard Gaussian prior $W$ (i.e.\ $\NNN(0, 1)$) on the parameter $\mu\in\RR$. Let $\tilde{\mu}(x^n) = {\EE}_{\mu^\star \sim W \mid X^n =x^n}\br{\mu^\star}$ be the posterior mean based on data $X^n = x^n$. Notice that 
\[
\tilde{\mu}(X^n)=\frac{n}{n+1}\MMLE(X^n) = \left(1-\frac1{n+1}\right) \sum_{i=1}^n X_i . 
\]
Let $\tilde\mu_n=\tilde\mu(X^n)$ and define the events $\{\cF_{\mu, n} \}_{n \in \naturals}$ as
$$\cF_{\mu, n} \Leftrightarrow (\mu - \tilde{\mu}_n)^2 \geq c \cdot f(n)
$$
for a fixed $c > 0$. In the next lemma we use the law of iterated logarithm to show that $\cF_{\mu, n}$ happens infinitely often with probability one.
\begin{lem}
\label{lem:lil}
 There exists $c>0$ such that for all $\mu \in \reals$, 
\begin{equation}\label{eq:physics}
\PP_{\mu}\left[ \cF_{\mu, n} \ \text{\ i.o.\ } \right] = 1. 
\end{equation}
\end{lem}
\begin{proof}
Assume that $\PP_{\mu}\left[ \cF_{\mu, n} \ \text{\ i.o.\ } \right] \neq 1. $ Then for every $c>0$ there exists $\mu\in\RR$ and $\delta>0$ such that 
\begin{equation}
\label{eq:as}
\PP_\mu\left[\exists\ N>0 \text{ s.t. } \forall n>N\, \, (\mu-\tilde\mu_n)^2\leq cf(n)\right]\geq\delta. 
\end{equation}
Fix $c>0$ and consider $\mu$ and $\delta$ such that \eqref{eq:as} is satisfied. Notice that
\[
|\mu-\MMLE_n|\leq \left|\mu-\MMLE_n+\frac1n\mu\right|+\frac1{n}|\mu|=\frac{n+1}{n}\left|\mu-\tilde\mu_n\right|+\frac1{n}|\mu|
\]
When $(\mu-\tilde\mu_n)^2\leq cf(n)$ and $n>\max\{3, \mu^2/c\}$ we have
\[
|\mu-\MMLE_n|\leq\frac{n+1}{n}\sqrt{cf(n)}+\frac1n|\mu|\leq \sqrt{3cf(n)}. 
\]
Therefore, 
\[
\PP_\mu\left[\exists\ N>0 \text{ s.t. } \forall n>\max\{3, \mu^2/c, N\}\, \, (\mu-\MMLE_n)^2\leq 3cf(n)\right]\geq\delta. 
\]
However, since $\MMLE=\frac1n\sum_{i=1}^n X_i$, we can apply the law of the iterated logarithm (see e. g. \cite{lil}), according to which for any $c'\in(0, 1)$
\[
\PP_\mu\left[\forall N>0\, \, \exists n>N \text{ s. t. } (\mu-\MMLE)^2\geq c'f(n)\right]=1. 
\]
Choosing $c\in(0, 1/3)$ leads to a contradiction, which proves the result. 

\end{proof}

Let $c$ be such that (\ref{eq:physics}) holds. Furthermore, let $g: \naturals \rightarrow \reals^+$ be any non-increasing function with $f(n) / g(n) \rightarrow \infty$. Finally, for a fixed estimator $\hat\mu$ define the events $\cG_n$ 
$$
\cG_n \ \Leftrightarrow \ (\tilde{\mu}(X^n) -\hat{\mu}(X^n))^2 \geq \frac{1}{2} \cdot c \cdot f(n). 
$$

In the following lemma we show that we only need to consider estimators for which $\PP_{\mu} ({\cG}_n \ \text{\rm i. o}) = 1$ holds for all $\mu\in\RR$. 
\begin{lem}
\label{lem:twocases}
 Let $\hat\mu$ be an arbitrary estimator. Then either (\ref{eq:friday0}) and (\ref{eq:friday}) holds, 
or
\begin{equation}\label{eq:sunday}
\text{for all $\mu\in\RR$:\ } \PP_{\mu} [{\cG}_n \ \text{\rm i.o. }] = 1. 
\end{equation}
\end{lem}
\begin{proof}
Fix arbitrary $\mu$ and let $\bar{\cG}_n$ be the complement of
$\cG_n$. We consider two cases, depending on whether
$\PP_{\mu}[\cF_{\mu, n} \cap \bar{\cG}_n \ \text{\rm i.o. }] > 0$ or not. 

\smallskip
\noindent{\bf Case 1:} $\PP_{\mu}[\cF_{\mu, n} \cap \bar{\cG}_n \ \text{\rm i.o. }] > 0$. \\
In this case, there exists $\epsilon > 0$ such that for all $n_0 \in \naturals$ there is an $n_1 > n_0$ such that
$$
\PP_{\mu}\left[ \cF_{\mu, n} \cap \bar{\cG}_n \ \text{holds for some $n$ with $n_0 < n < n_1$ }\right] \geq \epsilon. 
$$
Define the stopping time $\tau = \min \{n_1, n \}$ with $n_1$ as above and $n$ the smallest $n > n_0$ such that $\cF_{\mu, n} \cap \cG_n$ holds. Then $\tau$ is finite and we must have
 $$
 {\EE}_{X^\infty\sim \PP_{\mu}} \left[ \frac{(\hat\mu(X^{\tau}) - \mu)^2}{g(\tau)} \right]
 = 
 {\EE}\left[ \frac{(\hat\mu(X^{\tau}) - \mu)^2}{f(\tau)} \cdot
 \frac{f(\tau)}{g(\tau)} \right]
 \geq \epsilon \cdot \frac{1}{2} c \cdot \min_{n_0 < n < n_1} \frac{f(n)}{g(n)}. 
 $$
Since $f(n)/g(n) \rightarrow \infty$, and we can take $n_0$ as large as we like, it follows that in this case (\ref{eq:friday}) holds. Similarly, by taking $g=f$ we get (\ref{eq:friday0}).

\smallskip
\noindent{\bf Case 2:}
$P_{\mu}(\cF_{\mu, n} \cap \bar{\cG}_n \ \text{\rm i.o.}) = 0$. \\
This statement, together with (\ref{eq:physics}) directly implies (\ref{eq:sunday}); this finishes the proof the lemma. 
\end{proof}

\subsubsection{Translating the problem into bounding Bayes risk}
\label{br}
First, we state a general measure-theoretic result on the existence of the joint density of $(\tau, X^\tau)$. The proof can be found in the \hyperref[appn]{Appendix}.
\begin{prop}
\label{thm:bayes}
Let $\{\PP_\mu, \mu\in M\subseteq \RR\}$ be a set of probability measures on some space $(\XX, \BB, \nu)$ such that for all $\mu\in M$ there exists a conditional density function $p(\cdot\given \mu):\XX\to\RR^+$ with respect to $\nu$. For each $\mu\in M$ consider $X^n=(X_1, \dots, X_n)$ to be a vector of i.i.d.\ random variables distributed according to $\PP_\mu$. Let $\tau$ be an a.s.-finite stopping time with respect to filtration $\FF=(\FF_n=\sigma_\mu(X_1, \dots, X_n))_{n\in\NN}$. Let $Y=(\tau, X^\tau)$ be a random variable taking values in $\YY=\bigcup\limits_{n=1}^\infty \{n\}\times \XX^n$. Then there exists a $\sigma$-algebra $\Sigma$ and measure $\nu_\YY$ on $\YY$ such that for each $\mu \in M$ rv $Y$ is $\Sigma$-measurable and has a density function $p(\cdot\given\mu):\YY\to\RR^+$ with respect to $\nu_\YY$. 
\end{prop}
Then, using Bayes' theorem (see e.g. Theorem 1. 16 in \cite{schervish}) we get the following. 
\begin{cor}
\label{thm:b}
Consider any prior $W$ on $\RR$ with density $w(\mu)$. Let $W_{y}(\cdot\given y)$ denote the conditional distribution of $\mu$ given $Y=y$. Then, $W_y\ll W$ and there exists a conditional density $w(\mu\given n, x^n)$ such that for any $Y=(n, x^n)$
\[
w(\mu\given n, x^n)=\frac{w(\mu)p(n, x^n\given\mu)}{\int_M p(n, x^n\given\tilde\mu)w(\tilde\mu)d\tilde\mu}. 
\]
\end{cor}

Notice that for any estimator $\hat\mu_n=\hat\mu(X^n)$, stopping time $\tau^\star\in \TT$, prior $W$ on $\RR$, and a function $g:\NN\to\RR^+$
we have
\begin{equation}
\label{eq:sup}
\sup_{ \mu\in \RR}\sup_{\tau\in\TT}\EE_{X^\infty\sim\PP_ \mu}\br{\frac{( \mu-\hat \mu_{\tau})^2}{g(\tau)}}\geq
\EE_{ \mu\sim W}\EE_{X^\infty\sim\PP_ \mu}\br{\frac{(\mu-\hat \mu_{\tau^\star})^2}{g(\tau^\star)}} . 
\end{equation}

Denote $Y = (\tau^\star, X^{\tau^\star})\in\YY$ and $h(\mu, Y)= \frac{(\mu-\hat \mu_{\tau^\star})^2}{g(\tau^\star)}$. Let $w( \mu)$ be the density function of the prior $W=\NNN(0, 1)$. Then using the results of Corollary \ref{thm:b} we have
\begin{align*}
\EE_{\mu\sim W}\EE_{X^\infty\sim\PP_\mu} \br{h(\mu, Y)}=&
\int_{\RR}\int_\YY h(\mu, Y) p(Y\given\mu) w(\mu)dYd\mu=\\
=&\int_{\RR}\int_\YY h(\mu, Y) w(\mu\given Y)\int_{\RR} p(y\given\mu^\star)w(\mu^\star)d\mu^\star dYd\mu=\\
=&\int_{\RR}\int_\YY\int_\RR h(\mu, Y) w(\mu\given Y)d\mu p(y\given\mu^\star)dY w(\mu^\star)d\mu^\star=\\
=&\EE_{\mu^\star\sim W}\EE_{Y\sim \PP_{\mu^\star} }{\EE}_{\mu \sim W \mid Y} \br{h(\mu, Y)}=\\
=&\EE_{Y\sim\bar P} {\EE}_{\mu \sim W \mid Y} \br{h(\mu, Y)}, 
\end{align*}
where $\bar{P}$ is the Bayes marginal distribution based on the prior $W$.

Therefore, in order to prove the theorem we only need to show that for all estimators $\hat\mu$ that satisfy \eqref{eq:sunday} there exists a stopping time $\tau^\star$ such that for some $C>0$ (which will depend on $c$)
\smallskip

\noindent(i) for a function $f(n)=n^{-1}\log\log n$ (with $f(n)=1$ when $n=1, 2$)
\[
{\EE}_{Y \sim \bar{P}} {\EE}_{ \mu \sim W \mid Y} \br{\frac{(\mu-\hat \mu_{\tau^\star})^2}{f(\tau)}}\geq C. 
\]
(ii) for all non-increasing functions $g:\NN\to\RR^+$ such that $f(n)/g(n)\to\infty$
\[
{\EE}_{Y \sim \bar{P}} {\EE}_{ \mu \sim W \mid Y} \br{\frac{( \mu-\hat \mu_{\tau^\star})^2}{g(\tau)}}=\infty. 
\]

\subsubsection{Defining the suitable stopping time}
\label{sec:stop}
For a fixed estimator $\hat\mu$ that satisfies \eqref{eq:sunday} and a fixed $n_0>0$ define the stopping time $\tau^\star$ as 
\[
\tau^\star=\min\{n\in \RR: n> n_0 \text{ and } \cG_n \text{ holds}\}. 
\]
By (\ref{eq:sunday}) this stopping time is $P_{\mu}$-a.s. finite for all $\mu$. Let $Y=(\tau^\star, X^{\tau^\star})$. For every $Y=(n, x^n)$ we have

\begin{align*}
 {\EE}_{\mu \sim W \mid Y=(n, x^n)} \br{h(\mu, Y)} &= {\EE}_{\mu \sim W \mid X^n = x^n} \left[
 \frac{(\mu - \hat \mu_n)^2}{g(n)} 
 \right] = \\ &= {\EE}_{\mu \sim W \mid X^n = x^n} \left[
 \frac{(\mu - \tilde{\mu}_n)^2 + (\tilde{\mu}_n - \hat\mu_n)^2}{g(n)} 
 \right]\geq \\ 
 &\geq \frac{(\tilde\mu_n - \hat\mu_n)^2}{g(n)}.
 \end{align*}
The first equality holds due to the event $\{\tau=n\}$ be completely determined by the event $\{X^n=x^n\}$. The second equality is due to $\tilde{\mu}_n=\tilde{\mu}(X^n)$ being the posterior mean given $X^n$ based on prior $W$. 
Furthermore, by definition of $\tau^\star$, the vector $X^{\tau^\star}$ satisfies 
\[
{(\tilde\mu(X^{\tau^\star}) - \hat\mu(X^{\tau^\star}))^2}\geq (c/2) \cdot f(\tau^\star).
\] 
Then for every $Y=(n, x^n)$
$$
 {\EE}_{\mu \sim W \mid Y=(n, x^n)} \left[
 \frac{(\mu - \hat \mu_n)^2}{g(n)} 
 \right] \geq\frac c2 \cdot \min_{n > n_0} \frac{f(n)}{g(n)}. 
$$
Since we can choose $n_0$ arbitrarily large and $f(n)/g(n) \rightarrow \infty$, the desired results follows.

\begin{appendix}

\section*{Remaining proofs}\label{appn}
\noindent {\bf Proof of Lemma \ref{main}: $k>1$}

Let $\gamma_i=\frac1{i(i+1)}$, $\eta^i=c_0\sqrt{\frac{-\log\gamma_i}{e^i}}$ for some constant $c_0>0$ such that $(\eta^i)^2\leq\frac\delta k$ for all $i=1, \dots, \infty$. The above holds for all positive $c_0<\delta/k.$

Furthermore, let $P=\{\rho=(\rho_1, \dots, \rho_k)\in\{-1, 1\}^k, \text{ such that } \rho_1=1\}$. Notice that $|P|=2^{k-1}$. For a fixed $\rho\in P$ consider a discrete probability measure on $B_\delta(0)$ with density
\[
\gamma(\eta)=\sum_{i\in\NN} \gamma_i \ind{\eta=\eta_{i, \rho}}
\] 
with
$
\eta_{i, \rho}=(\eta^i\rho_1, \dots, \eta^i\rho_k)^T. 
$
Then, since $\eta_{i, \rho}^T\eta_{i, \rho}=k(\eta^i)^2$ we have
\[
Z_n=\sum_{i=1}^{\infty}\gamma_ie^{\eta_{i, \rho}^T S_n-n\sigma\eta_{i, \rho}^T\eta_{i, \rho}/2}\geq \max_{i\in\NN} \gamma_i e^{\eta^i T_{n, \rho}-n\sigma k(\eta^i)^2/2}, 
\]
where $T_{n, \rho}=S_n^1\rho_1+\dots+S_n^k\rho_k. $ Now we apply the argument from the one-dimensional case and get for any $\rho\in P$
\[
\EE \left[\sup_{n>27}e^{\frac{c_0\sqrt2}{4}\log\log n \left(\frac{|T_{n , \rho}|}{\sqrt{n \log\log n }}- \frac{\sqrt{2}}{2c_0}\left(9\sigma k c_0^2+6\right) \right)}\right]\leq 4. 
\]
Notice that $T_n=|S_n^1|+\dots+|S_n^k|=\max_{\rho\in P} |T_{n, \rho}|. $ Since $|P|=2^{k-1}$, we have
\begin{multline*}
\EE \left[\sup_{n>27} e^{\frac{c_0\sqrt2}{4}\log\log n \left(\frac{T_n }{\sqrt{n \log\log n }}- \frac{\sqrt{2}}{2c_0}\left(9\sigma k c_0^2+6\right) \right)}\right]=\\
=\EE \left[\sup_{n>27}\max_{\rho\in P} e^{\frac{c_0\sqrt2}{4}\log\log n \left(\frac{|T_{n , \rho}|}{\sqrt{n \log\log n }}- \frac{\sqrt{2}}{2 c_0}\left(9\sigma k c_0^2+6\right) \right)}\right]\leq\\
\leq\sum_{\rho\in P}\EE\left[\sup_{n>27} e^{\frac{ c_0\sqrt2}{4}\log\log n \left(\frac{|T_{n , \rho}|}{\sqrt{n \log\log n }}- \frac{\sqrt{2}}{2 c_0}\left(9\sigma k c_0^2+6\right) \right)}\right]\leq 2^{k+1}. 
\end{multline*}
The desired result follows by taking $c=\frac {c_0\sqrt2}{4}$ and the fact that $\log\log n >1$, when $n>27$. 

\medskip
\noindent {\bf Proof of Proposition \ref{thm:bayes}}

Consider $\YY=\bigcup\limits_{n=1}^\infty \{n\}\times \XX^n$. For each $A\subset\YY$ let
\[
A[n]=\{a\in\XX^n: (n, a)\in A\}. 
\]
We endow $\YY$ with a $\sigma$-algebra $\Sigma$ by setting $A\in\Sigma$, iff $A[n]\in\BB^n$. Here $\BB^n$ is a product $\sigma$-algebra on $\XX^n$. Clearly $\Sigma$ is a $\sigma$-algebra. For $A\in\Sigma$ let
\[
\nu_\YY(A)=\sum_{n=1}^\infty \nu_n(A[n]), 
\]
where $\nu_n=\nu^{\otimes n}$. Notice that $\nu$ is a $\sigma$-finite measure on $\YY$. 

Consider any $\mu\in M. $ For $A\in\Sigma$ let $\PP_{Y}[A\given \mu]=\PP\left[(\tau, X^\tau)\in A\given\mu\right]$. Notice that $\PP_Y[\emptyset\given \mu]=0$, $\PP_Y[\YY\given \mu]=1$, and $\PP_Y[A\given \mu]\in[0, 1]$ for every $A\in\Sigma$. Also, for a any countable collection $A_i\in\Sigma$ of pairwise disjoints sets we have
\[
\PP_Y\left[\bigcup_{i=1}^\infty A_i\given \mu\right]=\sum\PP_Y[A_i\given \mu]. 
\] 
Therefore $\PP$ is a probability measure. Furthermore, consider $A\in\Sigma$ such that $\nu_\YY(A)=0. $ Denote $A[n]=(A^1[n], \dots, A^n[n])$, where $A^i[n]\subset\XX$. Then
\[
\nu_\YY(A)=\sum_{n=1}^\infty\prod_{i=1}^n \nu(A^i[n])=0. 
\]
Thus, $\nu(A^i[n])=0$ for all $n\in\NN$ and $i\in\{1, \dots, n\}$. Since $\PP$ is absolutely continues with respect to $\nu$, we have $\PP[A^i[n]]=0$. Therefore, for each $n$
\[
\PP_Y[n, A[n]\given \mu]=\PP[\tau=n, X^n\in A[n]\given \mu]\leq\PP[X^n\in A[n]\given \mu]=\prod_{i=1}^n\PP[X_i\in A^i[n]\given \mu]=0. 
\]
Then $\PP_Y[A\given \mu]=\sum_{n=1}^\infty \PP_Y[n, A[n]\given \mu]=0. $ Therefore, $\PP_Y[\cdot\given\mu]$ is absolutely continuous with respect to $\nu_\YY$ for every $\mu\in M$. By the Radon-Nikodym theorem there exists a density $p_Y(n, x^n\given \mu)$ with respect to measure $\nu_\YY$. 

\end{appendix}
\section*{Acknowledgements}
This work is part of the research program {\it Safe Bayesian Inference} with project number 617. 001. 651 which is financed by the Dutch Research Council (NWO). The project leading to this work has received funding from the European Research Council (ERC) under the European Union’s Horizon 2020 research and innovation programme (grant agreement No 834175). We thank Wouter Koolen for several useful conversations.

\bibliographystyle{imsart-nameyear} 
%\scriptsize
\bibliography{references}

\end{document}